\documentclass[12pt]{article}
\usepackage{a4wide}

\usepackage{amsmath,amssymb,amsthm}
\usepackage[mathscr]{eucal}
\usepackage{graphics}
\usepackage[hypertex]{hyperref}
\usepackage[usenames]{color}

\definecolor{red}{rgb}{1,0.00,0.00}
\definecolor{blue}{rgb}{0,0.08,0.55}
\definecolor{green}{rgb}{0.0,0.6,0}

\author{Ievgen~V.~Bondarenko}
\title{\textbf{Finite generation of iterated wreath products}}

\newcommand{\at}{\mathop{\mbox{\textcircled{${{\ae}}$}}}}
\newcommand{\Sym}{\mathop{\rm Sym}\nolimits}
\newcommand{\Aut}{\mathop{\rm Aut}\nolimits}
\newcommand{\st}{\mathop{\rm St}\nolimits}
\newcommand{\rs}{\mathop{\rm RiSt}\nolimits}
\newcommand{\tree}{\mathcal{T}}

\newtheorem{thm}{Theorem}

\newtheorem{cor}[thm]{Corollary}

\theoremstyle{definition}

\newtheorem{ex}{Example}
\newtheorem{rem}{Remark}

\begin{document}
\maketitle

\begin{abstract}
Let $(G_n,X_n)$ be a sequence of finite transitive permutation groups with uniformly bounded number
of generators. We prove that the infinitely iterated permutational wreath product $\ldots\wr G_2\wr
G_1$ is topologically finitely generated if and only if the profinite abelian group $\prod_{n\geq
1} G_n/G'_n$ is topologically finitely generated. As a corollary, for a finite transitive group $G$
the minimal number of generators of the wreath power $G\wr \ldots\wr G\wr G$ ($n$ times) is bounded
if $G$ is perfect, and grows linearly if $G$ is non-perfect. As a by-product we construct a
finitely generated branch group, which has maximal subgroups of infinite index, answering
\cite[Question~14]{branch_groups}.\\


\end{abstract}


\section{\large Introduction}

Let $(G,X)$ and $(H,Y)$ be permutation groups (all actions in the paper are faithful). The
permutational wreath product $H\wr G$ is the semidirect product $H^{|X|}\rtimes G$ with a natural
action on $X\times Y$, where $G$ acts on $H^{|X|}$ by permuting the copies of~$H$. Given a sequence
of finite permutation groups $(G_n,X_n)$ the inverse limit of iterated wreath products
\[
W=\lim_{\longleftarrow} \ (G_n\wr\ldots \wr G_2\wr G_1)
\]
is a profinite group called the infinitely iterated wreath product $\ldots\wr G_2\wr G_1$. The goal
of this note is to prove the following theorem.

\begin{thm}\label{thm_main}
Let $(G_n,X_n)$ be a sequence of finite transitive permutation groups with uniformly bounded number
of generators. Then the profinite group $\ldots\wr G_2\wr G_1$ is finitely generated if and only if
the profinite abelian group $\prod_{n\geq 1} G_n/G'_n$ is finitely generated.
\end{thm}

Since the group $\prod_{n\geq 1} G_n/G'_n$ is an epimorphic image of the group $\ldots\wr G_2\wr
G_1$, in one direction the statement is obvious. For the converse we construct a finitely generated
dense subgroup. The construction is based on the notions of directed automorphisms and branch
groups introduced by R.~Grigorchuk \cite{branch:gri}, and it is basically the same construction
used by P.~Neumann \cite{neumann86}, L.~Bartholdi \cite{uniform:bart}, D.~Segal \cite{fin_im:segal}
and others to construct certain groups with interesting properties.

M.~Bhattacharjee \cite{bhattach} showed that the infinitely iterated wreath products of alternating
groups of degree $\geq 5$ can be generated by two elements even with positive probability. This was
generalized to wreath products of non-abelian simple groups with transitive actions by M.~Quick
\cite{quick1,quick2}. D.~Segal \cite{fin_im:segal} showed that the infinitely iterated wreath
products of perfect groups with certain conditions on the actions are finitely generated. The
iterated wreath products of cyclic groups of pairwise coprime orders are two-generated by result of
A.~Woryna \cite{woryna}.

The profinite group $W$ is finitely generated if and only if the sequence $d(G_n\wr \ldots \wr
G_2\wr G_1)$ is bounded, where $d(G)$ is the minimal number of generators of the group $G$. The
behavior of generating sequence $d(G^n)$ for the direct product $G^n$ of a finite group $G$ was
described in a series of papers by J.~Wiegold (see \cite{growth_seqII,growth_fgg} and references
therein): the sequence $d(G^n)$ is roughly logarithmic if $G$ is perfect, and is linear otherwise.
The generating sequence $d_n^{wr}(G)=d(G\wr\ldots\wr G\wr G)$ for the wreath power of a finite
transitive group $G$ happens to be even more simple. As a corollary of Theorem~\ref{thm_main} we
get that $d_n^{wr}(G)$ is bounded if $G$ is perfect, and grows linearly otherwise.

The wreath powers of finite groups are related to self-similar groups of finite type introduced by
R.~Grigorchuk \cite[Section~7]{solved:gri} to describe profinite completion of certain branch
groups. Every such group is given by a finite pattern, and there is an open question
\cite[Problem~7.3]{solved:gri}: which patterns define finitely generated groups. Every self-similar
group of finite type given by pattern of size $1$ is isomorphic to the infinite wreath power
$\ldots\wr G\wr G$. Hence it will be finitely generated if and only if the group $G$ is perfect.
Moreover, in this case the finitely generated dense subgroup constructed in the proof of
Theorem~\ref{thm_main} is a branch group, which has maximal subgroups of infinite index. This
answers Question~14 in \cite[p.~1107]{branch_groups}.\\[-0.3cm]

\noindent\textbf{Acknowledgments.} I would like to thank Rostyslav Kravchenko and Dmytro Savchuk
for fruitful discussions.

\section{\large Automorphisms of spherically homogeneous rooted tree}

Let $\tree$ be the spherically homogeneous rooted tree given by the alphabets $X_i$, where the
$n$-th level is $\tree_n=X_1\times\ldots\times X_n$ (here $\tree_0$ consists of the root
$\emptyset$), and each vertex $v\in \tree_n$ is connected with $vx\in \tree_{n+1}$ for every $x\in
X_{n+1}$. Denote by $\tree^{[n]}$ the finite truncated rooted tree consisting of levels from $0$ to
$n$. For a vertex $v\in\tree$ denote by $\tree_v$ the rooted tree hanging ``below'' the vertex $v$,
which consists of vertices $u$ such that the geodesic connecting $u$ with the root $\emptyset$
passes through~$v$. The set $vX_n$ is the first level of tree $\tree_v$ for $v\in\tree_{n-1}$.

Every automorphism $g\in\Aut\tree$ induces a map $vX_n\rightarrow g(v)X_n$ for every vertex
$v\in\tree_{n-1}$. Forgetting the prefixes $v$ and $g(v)$ we get a permutation on the set $X_n$,
which is called the \textit{vertex permutation of $g$ at $v$} and denoted $g \at v\in\Sym(X_{n})$.
Every automorphism can be given by its vertex permutations at each vertex.

The automorphism group of the tree $\tree^{[n]}$ is the iterated permutational wreath product
\[
\Aut\tree^{[n]}=\Sym(X_{n})\wr \ldots\wr \Sym(X_1),
\]
and the group $\Aut\tree$ is the inverse limit of these groups, which is also the infinite
permutational wreath product $\ldots\wr\Sym(X_2)\wr\Sym(X_1)$. The group $\ldots\wr G_2\wr G_1$ can
be naturally considered as a subgroup of $\Aut\tree$. We also identify the group $G_1$ with the
subgroup of $\Aut\tree$, which consists of automorphisms whose vertex permutations at the root form
the group $G_1$ and are trivial at the other vertices.

Let $G$ be a subgroup of $\Aut\tree$. The \textit{vertex stabilizer} $\st_G(v)$ of a vertex
$v\in\tree$ is the subgroup of all $g\in G$ such that $g(v)=v$. The \textit{$n$-th level
stabilizer} $\st_G(n)$ is the subgroup of all $g\in G$ such that $g(v)=v$ for every
$v\in\tree_{n}$. The \textit{rigid vertex stabilizer} $\rs_G(v)$ of a vertex $v\in\tree$ is the
subgroup of all $g\in G$ such that the vertex permutations of $g$ at vertices outside the subtree
$\tree_v$ are trivial. The \textit{rigid level stabilizer} $\rs_G(n)$ is the subgroup generated by
$\rs_G(v)$ for $v\in\tree_n$. The group $G$ is called \textit{branch}, if it acts transitively on
the levels $\tree_n$ and every rigid level stabilizer $\rs_G(n)$ is of finite index in $G$.

We say that the rigid stabilizer $\rs_G(v)$ for $v\in\tree_{n-1}$ contains a subgroup
$H\in\Sym(X_n)$ as a \textit{rooted subgroup}, if $\rs_G(v)$ contains a subgroup whose vertex
permutations at $v$ form the group $H$ and all the other vertex permutations are trivial.

\section{\large  Proof of Theorem~\ref{thm_main} and Corollaries}

\begin{proof}(of Theorem~\ref{thm_main})
The trivial groups $G_n$ can be omitted and we can assume $|X_n|\geq 2$ for all $n$. Also we can
assume that every group $G_n$ is non-abelian, otherwise we can pass to a new sequence given by a
different arrangement of brackets $\ldots\wr(G_6\wr G_5)\wr (G_4\wr G_3)\wr (G_2\wr G_1)$ for which
it is true (the permutational wreath product is associative). By the same reason (passing to the
same arrangement of brackets with already non-abelian groups) we can assume that for every action
$(G'_n,X_n)$ there are points in the same orbit with different stabilizers. Fix letters $x_n,y_n\in
X_n$ and permutations $\tau_n,\pi_n\in G'_n$ such that $\tau_n(x_n)=y_n$, $\pi_n(x_n)=x_n$, and
$\pi_n(y_n)\neq y_n$.

Consider the decomposition of the groups $G_n/G'_n$ is the direct sum of cyclic groups of
prime-power order. Since the group $\prod_{n\geq 1} G_n/G'_n$ is finitely generated, there is an
absolute bound on the number of cyclic $p$-groups in these decomposition for any particular prime
number $p$. Hence there is a finite generating set $a_1=(a_1^{(n)})_n,\ldots, a_e=(a_e^{(n)})_n$
such that for every fixed $i$ two elements $a_i^{(j)}$ and $a_i^{(j')}$ have coprime orders for
every $j,j'$. Let $g_i^{(n)}\in G_n$ be a preimage of $a_i^{(n)}$ under the canonical projection
$G_n\rightarrow G_n/G'_n$. Then $G_n=\langle g_1^{(n)},\ldots, g_e^{(n)},G'_n\rangle$ for every
$n$. Since there is a uniform bound on the number of generators of $G_n$, we can complete the
elements $g_1^{(n)},\ldots, g_e^{(n)}$ to a generating set of $G_n$ by some elements
$g_{e+1}^{(n)},\ldots, g_m^{(n)}$, i.e. $G_n=\langle g_1^{(n)},\ldots, g_m^{(n)}\rangle$ for every
$n$ with fixed $m$. Define the automorphisms $g_1,\ldots, g_m$ of the tree $\tree$ by their vertex
permutations
\begin{equation}\label{eq_thm_defi_gi}
g_i\at v=\left\{
           \begin{array}{ll}
             g_i^{(k+1)}, & \hbox{if $v=y_1y_2\ldots y_{k-1}x_k$;} \\
             1, & \hbox{otherwise.}
           \end{array}
         \right.
\end{equation}

Consider the group $G=\langle g_1^{(1)},\ldots,g_m^{(1)},g_1,\ldots,g_m\rangle$ and let us prove
that $G$ is dense in the group $W$. Fix $n$ and let us show that the restriction of $G$ on the tree
$\tree^{[n]}$ coincides with the group $G_{n}\wr\ldots\wr G_2\wr G_1$.

Since the group $G$ acts transitively on the set $X_1$ and the vertex stabilizer $\st_G(y_1\ldots
y_{n-2}x_{n-1})$ acts transitively on $y_1\ldots y_{n-2}x_{n-1}X_n$ for every $n$, inductively we
get that the group $G$ acts transitively on the levels of $\tree$.

Let us prove that the rigid vertex stabilizers $\rs_G(y_1\ldots y_{k-1}y_k)$ and $\rs_G(y_1\ldots
y_{k-1}x_k)$ contain the group $G'_{k+1}$ as a rooted subgroup. Since elements $g_1^{(1)},\ldots,
g_m^{(1)}$ generate the group $G_1$ as a rooted subgroup at the root of the tree, the statement
holds at zero level. Assume that we have proved it for all levels $<k$. By inductive hypothesis
there exist $\pi,\tau\in G$ such that $\pi\at {y_1\ldots y_{k-1}}=\pi_k$, $\tau \at {y_1\ldots
y_{k-1}}=\tau_k$, and all the other vertex permutations are trivial. Then
\begin{eqnarray*}
\pi^{-1}g_i\pi \at {y_1\ldots y_{k-1}x_k}=g_i^{(k+1)},\quad \pi^{-1}g_i\pi \at {y_1\ldots
y_{s-1}x_s}=1\quad\mbox{ for } \ s>k \\ \Rightarrow\qquad [{\pi^{-1}g_i\pi,g_j}] \at {y_1\ldots
y_{l-1}x_l}=[g_i^{(l+1)},g_j^{(l+1)}]\quad\mbox{ for } \ l\leq k
\end{eqnarray*}
and all the other vertex permutations of $[\pi^{-1}g_i\pi,g_j]$ are trivial for all $g_i,g_j$. By
inductive\linebreak hypothesis, we can multiply $[\pi^{-1}g_i\pi,g_j]$ on the appropriate elements
from $\rs_G(y_1\ldots y_{l-1}x_l)$ for $l<k$ to remove all commutators at vertices $y_1\ldots
y_{l-1}x_l$, and get elements from $\rs_G(y_1\ldots y_{k-1}x_k)$. Conjugating by generators
$g_1,\ldots,g_m$ we get that $\rs_G(y_1\ldots y_{k-1}x_k)$ contains $G'_{k+1}$ as a rooted
subgroup. Conjugating by $\tau$ we get that $\rs_G(y_1\ldots y_{k-1}y_k)$ contains $G'_{k+1}$ as a
rooted subgroup.

The elements $a_i^{(1)},\ldots,a_i^{(n)}$ have pairwise coprime orders (here $i\leq e$), and for a
fixed $k\leq n$ we can choose a power $\alpha$ such that $(g_i)^{\alpha} \at {y_1\ldots
y_{k-2}x_{k-1}}=g_i^{(k)}h_i$ for some $h_i\in G'_{k}$, $(g_i)^{\alpha} \at {y_1\ldots
y_{l-2}x_{l-1}}\in G'_l$ for all $l\leq n$, $l\neq k$, and all the other vertex permutations at the
vertices of $\tree^{[n]}$ are trivial. Multiplying $(g_i)^{\alpha}$ on the corresponding elements
from $\rs_G(y_1\ldots y_{l-2}x_{l-1})$ we can remove the elements from commutants, and get
element\linebreak $f_i\in G$ such that $f_i \at {y_1\ldots y_{k-2}x_{k-1}}=g_i^{(k)}$ with all the
other vertex permutations at the vertices of $\tree^{[n]}$ being trivial. It follows that the group
$G$ contains a subgroup $H$ such that $H \at y_1\ldots y_{k-2}x_{k-1}=G_k$ and the vertex
permutations of every element of $H$ at the vertices of $\tree^{[n]}$ and outside $\tree_{y_1\ldots
y_{k-2}x_{k-1}}$ are trivial. Since the action is transitive this holds for every vertex
$v\in\tree_{k-1}$. The result follows.
\end{proof}

\begin{rem}
We got the bound on the number of generators. In some cases one can reduce the number of generators
as it was done in \cite{fin_im:segal} using Lemma~2 there.
\end{rem}

The next corollary is a generalization of result in \cite{woryna}.

\begin{cor}\label{cor_abel}
For a sequence $A_n$ of finite abelian transitive groups with pairwise coprime orders and uniformly
bounded number of generators the iterated wreath product $\ldots\wr A_2\wr A_1$ is finitely
generated.
\end{cor}

The last corollary also follows from Theorem~2 in \cite{gen_wreath_aug}, which says that if $(G,X)$
is a finite transitive permutation group and $H$ is a finite solvable group, then
\begin{equation}\label{eq_solvable1}
d(H\wr G)=\max\left(d(H/H'\wr G), \left[\frac{d(H)-2}{|X|}\right]+2\right).
\end{equation}
Moreover, by Corollary~6 in \cite{gen_wreath_aug}, if $H$ is abelian of coprime order with $|G|$,
then
\begin{equation}\label{eq_solvable2}
d(H\wr G)=\max\left(d(G), d(H)+1\right).
\end{equation}
Hence $d(A_n\wr\ldots\wr A_2\wr A_1)=\max_{2\leq i\leq n} (d(A_1), d(A_i)+1)$ under conditions of
Corollary~\ref{cor_abel}.

The next example shows that we cannot remove assumption on the number of generators in
Theorem~\ref{thm_main}.
\begin{ex}
Take a sequence of finite solvable transitive groups $(G_n,X_n)$ such that the groups $G_n/G'_n$
are cyclic with pairwise coprime orders and $d(G_n)$ is greater than $n|X_1||X_2|\cdots|X_{n-1}|$.
Then the sequence $d(G_n\wr\ldots\wr G_2\wr G_1)$ is not bounded by (\ref{eq_solvable1}), the group
$W$ is not finitely generated, while the group $\prod_{n\geq 1} G_n/G'_n$ is procyclic.
\end{ex}

The next examples show that the uniform bound on the number of generators in not necessary for the
group $W$ to be finitely generated.
\begin{ex}
Take a sequence of finite solvable transitive groups $(G_n,X_n)$ such that every group $G_n/G'_n$
is cyclic of order coprime with $G_{n-1}\wr\ldots\wr G_2\wr G_1$, and
$|X_1||X_2|\cdots|X_{n-1}|>d(G_n)\rightarrow\infty$. Then the sequence $d(G_n\wr\ldots\wr G_2\wr
G_1)$ is bounded by (\ref{eq_solvable1}) and (\ref{eq_solvable2}), hence the group $W$ is finitely
generated.
\end{ex}

\begin{ex}
Using result in \cite{gen_wreath_two} one can construct a sequence of finite perfect transitive
groups $(G_n,X_n)$ for $n\geq 2$ with the following properties: there are generating sets
$G_n=\langle g_1^{(n)},\ldots,g_{m_n}^{(n)}\rangle$ such that the elements $g_i^{(j)}$ have
pairwise coprime orders for all $i,j$, and $|X_{n-1}|>d(G_n)=m_n\rightarrow\infty$ when
$n\rightarrow\infty$. Take the cyclic group $G_1=\langle a\rangle$ of order $d(G_2)+1$ with the
regular action $(G_1,X_1)$. Fix different letters $x_n^{(1)},\ldots,x_{n}^{(m_n)},y_n\in X_n$ for
every $n$ and define the automorphism $b$ of the tree $\tree$ by its vertex permutations
\[
b \at v=\left\{
        \begin{array}{ll}
          g_i^{(n+1)}, & \hbox{if $v=y_1y_2\ldots y_{n-1}x_n^{(i)}$;} \\
          1, & \hbox{otherwise.}
        \end{array}
      \right.
\]
\end{ex}
Since the elements $g_i^{(j)}$ have coprime orders we can take a power of $b$ to get any
$g_i^{(j)}$ rooted at vertex $y_1\ldots y_{j-2}x_{j-1}^{(i)}$ with all the other vertex
permutations at vertices of $\tree^{[n]}$ being trivial for every fixed $n$. Since the group
$\langle a,b\rangle$ acts transitively on the levels of $\tree$, conjugating we get the group $G_j$
as a rooted subgroup at every vertex of $(j-1)$-th level with all the other vertex permutations at
vertices of $\tree^{[n]}$ being trivial. Hence the group $\langle a,b\rangle$ is dense in $W$,
while the abelianization of $W$ is the finite cyclic group $G_1$ and $d(G_n)\rightarrow\infty$.


\begin{cor}
Let $H$ be a finite transitive group. The infinite wreath power $\ldots\wr H\wr H$ is finitely
generated if and only if the group $H$ is perfect.
\end{cor}

Consider the previous corollary in more details. Let $H=\langle h_1,\ldots,h_m\rangle$ (we assume
the group $H$ satisfies the first paragraph of the proof of Theorem~\ref{thm_main},\linebreak
otherwise take $H\wr H$), then the dense group $G$ constructed in the proof of\linebreak
Theorem~\ref{thm_main} is generated by $h_1,\ldots, h_m$ and $g_1,\ldots,g_m$, where every
automorphism $g_i$ is defined recursively by
\[
g_i(v)=\left\{
         \begin{array}{ll}
           xg_i(u), & \hbox{$v=xu$;} \\
           yh_i(u), & \hbox{$v=yu$;} \\
           v, & \hbox{otherwise.}
         \end{array}
       \right.
\]
($x=x_1$ and $y=y_1$ are from the theorem). Here $\langle g_1,\ldots,g_m\rangle\simeq H$, and hence
the group $G$ is perfect, because it is generated by perfect subgroups. In the same way as in the
proof of Theorem~\ref{thm_main} we get that $\st_G(n)=\rs_G(n)\simeq G\times\ldots\times G$. In
particular, the group $G$ is branch (actually regular branch over itself, see definition in
\cite{branch:gri}). It is also just-infinite by \cite[Theorem~3]{branch:gri}, and satisfies the
congruence subgroup property by \cite[Proposition~2]{branch:gri}, i.e. every subgroup of finite
index contains the stabilizer of some level. In particular, the group $W=\ldots \wr H\wr H$ is not
only the closure of the group $G$ in $\Aut\tree$, but also the profinite completion of~$G$.

Consider the subgroup $F<G$, which consists of all automorphisms from $G$ that have only finitely
many non-trivial vertex permutations (finitary automorphisms). Since $\rs_G(v)$ contains $H$ as a
rooted subgroup for every vertex $v$, the group $F$ consists of all automorphisms of the tree
$\tree$, which have only finitely many non-trivial vertex permutations and all vertex permutations
are from the group $H$. In particular, the subgroup $F$ is dense in the groups $W$ and $G$. Also
$F$ is a proper subgroup of $G$ (here $g_1,\ldots,g_m$ are not in $F$). Since the group $G$ is
finitely generated, the subgroup $F$ is contained in some maximal subgroup $M$, which is also dense
in $G$ and hence of infinite index. Indeed, if $M$ had finite index then it would contain some
level stabilizer $\st_G(n)$; but $M/\st_G(n)=G/\st_G(n)$ and we get $M=G$, which contradicts
maximality of $M$. This answers Question~14 in\linebreak \cite[p.~1107]{branch_groups} (see
discussion in \cite[Section~6]{solved:gri}).


\begin{cor}
Let $H$ be a finite transitive group and put $d_n^{wr}(H)=d(H\wr \ldots \wr H\wr H)$. The sequence
$d_n^{wr}(H)$ is bounded if $H$ is perfect, and grows linearly otherwise.
\end{cor}
\begin{proof}
As shown above, if $H$ is perfect then
\[
d(H)\leq d_n^{wr}(H)\leq 2d(H\wr H).
\]

Let us show that if $H$ is non-perfect, then
\[
nd(H/H')\leq d^{wr}_n(H)\leq 2d(H\wr H\wr H\wr H)+nd(H/H')
\]
(here $H\wr H\wr H\wr H$ is to make sure that the action of commutant satisfies the first paragraph
of the proof of Theorem~\ref{thm_main}). The lower bound is clear. For the upper bound we proceed
as in the proof of Theorem~\ref{thm_main}. If $H=\langle h_1,\ldots,h_m\rangle$ then we define the
automorphisms $g_1,\ldots,g_m$ of the tree $\tree^{[n]}$ by (\ref{eq_thm_defi_gi}). For every
generator $a$ of $H/H'$ take its preimage $b$ in $H$ and construct $n-1$ automorphisms
$b^{(1)},\ldots,b^{(n-1)}$, where $b^{(i)}\at v=b$ if $v=y\ldots yx$ (here the length $=i$) and
trivial otherwise. The group generated by all constructed elements is isomorphic to $H\wr\ldots\wr
H\wr H$.
\end{proof}

\bibliographystyle{plain}

\end{document}